\documentclass[12pt]{article}
\usepackage[a4paper]{geometry}
\usepackage{amsmath, amssymb, amsthm, blkarray}
\usepackage{relsize, multirow}
\usepackage{subcaption, float}
\usepackage{tikz}
\usepackage{tikz-cd}
\usepackage{hyperref}

\tikzstyle{vertex}=[circle, draw, inner sep=0pt, minimum size=6pt]
\newcommand{\vertex}{\node[vertex]}

\newtheorem{thm}{Theorem}[section]
\newtheorem{theorem}[thm]{Theorem}
\newtheorem{lemma}[thm]{Lemma}
\newtheorem{prop}[thm]{Proposition}

\newtheorem{cor}[thm]{Corollary}

\newtheorem*{thm*}{Theorem}
\newtheorem*{lemma*}{Lemma}
\newtheorem*{prop*}{Proposition}
\newtheorem*{cor*}{Corollary}
\newtheorem*{conj*}{Conjecture}

\theoremstyle{definition}
\newtheorem{defn}[thm]{Definition}
\newtheorem{definition}[thm]{Definition}

\newtheorem{ex}[thm]{Example}
\newtheorem{remark}[thm]{Remark}

\newtheorem*{ques*}{Question}

\DeclareMathOperator{\rank}{rank}
\DeclareMathOperator{\cl}{cl}

\usepackage{graphicx} 
\usepackage{url}
\usepackage{color}
\title{\textbf{Representability of Flag Matroids}}
\author{
Daniel Irving Bernstein \\
\normalsize{Mathematics Department} \\
\normalsize{Tulane University}\\
\normalsize{New Orleans, U.S.A.}\\
\normalsize{\texttt{dbernstein1@tulane.edu}}
\and
Nathaniel Vaduthala \\
\normalsize{Mathematics Department} \\
\normalsize{Tulane University}\\
\normalsize{New Orleans, U.S.A.}\\
\normalsize{\texttt{nvaduthala@tulane.edu}}
}
\date{}
\begin{document}
\maketitle

\begin{abstract}
    We provide a new axiom system for flag matroids, characterize representability of uniform flag matroids, and give forbidden minor characterizations of full flag matroids that are representable over $\mathbb{F}_2$ and $\mathbb{F}_3$ along with regular full flag matroids. We also provide different equivalent characterizations for regular full flag matroids. 
\end{abstract}

\section{Introduction}
Matroids were introduced in the 1930s as a combinatorial abstraction of the notion of linear independence in vector spaces.
Their conception is often attributed to a 1935 paper of Whitney~\cite{whitney1935abstract}
but they were also contemporaneously developed by Nakasawa~\cite{nishimura2009lost}.
Since then, matroid theory has developed into a rich area of combinatorics
with deep relevance in many seemingly unrelated areas ranging from optimization to algebraic geometry.
Much research within matroid theory has been motivated by the search for elegant combinatorial descriptions of graphic and representable matroids.
In broad terms, the goal of this paper is to expand that line of research into a generalization called \emph{flag matroids}.

Just as a matroid is a combinatorial abstraction of a linear subspace of a vector space, a flag matroid is a combinatorial abstraction of a sequence of nested linear subspaces of a vector space.
Flag matroids are typically defined to be a sequence of matroids on the same ground set
satisfying a particular compatibility condition.
They also have other equivalent cryptomorphic definitions~\cite{borovik2003coxeter}.
Our first contribution is a new cryptomorphic axiom system for flag matroids in terms of what we call \emph{feasible sets}, borrowing terminology and ideas from the theory of greedoids.
We then characterize representability for uniform flag matroids,
and give forbidden-minor classifications for full flag matroids that are
$\mathbb{F}_2$-representable, $\mathbb{F}_3$-representable, and regular. We then provide equivalent characterizations for regular full flag matroids.

Flag matroids were first studied in the 1960s and 70s as sequences of strong maps in~\cite{higgs1968strong,kennedy1975majors,cheung1976combinatorial,kung1977core}.
They have connections to the K-theory of flag varieties~\cite{cameron2017flag,dinu2021k,jarra2024flag}
and have a rich interplay with other combinatorial structures~\cite{de2007natural,benedetti2024lattice,benedetti2019quotients,fujishige2022compression}.
\emph{Gaussian elimination greedoids}, also called \emph{Gauss greedoids},
are a class of flag matroids that have been studied in the greedoid theory literature~\cite{greedoids}
and were recently shown~\cite{grinberg2021greedoid,grinberga2020bhargava} to have relevance to Barghava's theory of P-orderings~\cite{bhargava1997p}.
A matroid lift is a particular kind of flag matroid that is of fundamental importance in matroid theory itself~(see e.g.~\cite[Chapter~7]{oxley}) and has recently been used to study matroid representability~\cite{bernstein2024matroid} and rigidity theory~\cite{bernstein2022generic,clinch2021abstract}.

The paper is organized as follows.
Section~\ref{section: matroid background} covers the necessary matroid theory background.
Section~\ref{section: flag matroids} begins with some background on flag matroids.
We then provide a new cryptomorphic axiom system for flag matroids in Definition~\ref{defn: flag matroid cryptomorphism} that can be seen as a simultaneous generalization of the independent set, basis, and spanning set axioms of a matroid~(see Remark~\ref{remark: flag matroid axioms genearlize matroid axioms}).
We characterize the fields that each uniform flag matroid is representable over in Theorem~\ref{thm: uniform flag matroids representability}.
We review minors and duality for flag matroids and describe how these concepts manifest in our new axiom system.
It was shown in~\cite{kung1986strong} that to every flag matroid, one can associate a certain matroid called a \emph{major} that encodes the flag matroid in certain minors.
We recall this theory of majors, and show in Theorems~\ref{thm: representable iff has representable major} and~\ref{thm: graphic iff major is} that a flag matroid is graphic/$\mathbb{K}$-representable if and only if it has a major which is as well.
The main results of Section~\ref{section: representability} are Theorem~\ref{thm: full binary forbidden minors}, which gives excluded minor characterizations of $\mathbb{F}_2$- and $\mathbb{F}_3$-representable full flag matroids, and Theorem~\ref{thm: regular flag matroids forbidden minors}, which gives excluded minor characterizations of regular full flag matroids.

\section{Matroid theory background}\label{section: matroid background}

We begin with the minimal necessary background on matroid theory;
for a more leisurely and comprehensive introduction, see~\cite{oxley}.

\begin{definition}\label{defn: matroid}
    A matroid is a pair $M = (E, \mathcal{I})$, consisting of a finite set $E$ and a collection $\mathcal{I}$ of subsets of $E$ satisfying
\begin{enumerate}
    \item $\emptyset \in \mathcal{I}$ 
    \item for each $I' \subseteq I$, if $I \in \mathcal{I}$, then $I' \in \mathcal{I}$.
    \item for all $I, J \in \mathcal{I}$ such that $|J| > |I|$, there exists some $j \in J \setminus I$ such that $I \cup \{j\} \in \mathcal{I}$.
\end{enumerate}
Here $E$ is called the \emph{ground set} and elements of $\mathcal{I}$ are called \emph{independent sets}.
\end{definition}

Definitions~\ref{defn: uniform matroids} and~\ref{defn: representable matroids} below each give a family  of matroids.

\begin{defn}\label{defn: uniform matroids}
    Let $0 \le r \le n$ be integers, let $E$ be a set of size $n$ and let $\mathcal{I}$ consist of all subsets of $E$ of cardinality $r$ or less.
    Then $(E,\mathcal{I})$ is a matroid, denoted $U_{r,n}$.
    Matroids of the form $U_{r,n}$ are called \emph{uniform}.
\end{defn}

\begin{defn}\label{defn: representable matroids}
    Let $\mathbb{K}$ be a field and let $A$ be a matrix with entries in $\mathbb{K}$.
    If $E$ is (a set in natural bijection with) the column set of $A$ and $\mathcal{I}$ denotes the subsets of $E$ that are linearly independent,
    then $(E,\mathcal{I})$ is a matroid which we denote $M(A)$.
    Matroids arising in this way are called \emph{$\mathbb{K}$-representable}
    and a matrix $A$ such that $M = M(A)$ is called a \emph{$\mathbb{K}$-representation} of $M$.
\end{defn}

We now give two examples illustrating Definitions~\ref{defn: uniform matroids} and~\ref{defn: representable matroids}.

\begin{ex}\label{ex: representability of U24} 
Let $\mathbb{K}$ be a field with at least three elements and let $x \in \mathbb{K}\setminus\{0,1\}$.
Then the following is a $\mathbb{K}$-representation of $U_{2,4}$
\[
    \begin{pmatrix}
        0 & 1 & 1 & 1 \\
        1 & 0 & 1 & x
    \end{pmatrix}.
\]
Thus $U_{2,4}$ is $\mathbb{K}$-representable whenever $\mathbb{K}$ has three or more elements.
For each prime power $q$, the field with $q$ elements will be denoted $\mathbb{F}_q$.
It is relatively straightforward to show that $U_{2,4}$ is \emph{not} representable over $\mathbb{F}_2$.
\end{ex}

\begin{ex}\label{ex: fano plane}
    Consider the following matrix with entries in $\mathbb{F}_2$
    \[
        A =     
        \begin{pmatrix}
            1 & 1 & 1 & 1 & 0 & 0 & 0 \\
            1 & 1 & 0 & 0 & 1 & 1 & 0 \\
            1 & 0 & 1 & 0 & 1 & 0 & 1
        \end{pmatrix} \in \mathbb{F}_2^{3\times 7}.
    \]
    The matroid $M(A)$, often called the \emph{Fano matroid} and denoted $F_7$,
    is representable over $\mathbb{K}$ if and only if the characteristic of $\mathbb{K}$ is two~\cite[Proposition 6.4.8]{oxley}.
\end{ex}

    We now describe an alternative view of representable matroids.
    Let $\mathbb{K}$ be a field and let $V$ be a finite dimensional $\mathbb{K}$-vector space with basis $E$.
    For each linear subspace $L \subseteq V$, let $\mathcal{I}$ denote the subsets $I \subseteq E$
    such that the orthogonal projection of $L$ onto the linear space spanned by $I$ has dimension $|I|$.
    Then $(E,\mathcal{I})$ is a matroid which we denote $M(L)$.
    Matroids arising in this way are exactly the $\mathbb{K}$-representable matroids -
    if $A$ is any matrix whose rows represent a spanning set of $L$ in the basis $E$,
    then $M(A) = M(L)$.

The last family of examples of matroids we need to introduce come from graphs.
The graph with vertex set $V$ and edge set $E$ will be denoted $(V,E)$.
\begin{defn}\label{defn: graphic matroid}
    Let $G = (V,E)$ be a graph with loops and multiple edges allowed, and consider the family $\mathcal{I}$ of subsets of $E$ defined as follows
    \[
        \mathcal{I} := \{I \subseteq E : (V,I) {\rm \ has \ no \ cycles} \}.
    \]
    Then $(E,\mathcal{I})$ is a matroid, denoted $M(G)$.
    If $M$ is a matroid such that $M = M(G)$ for a graph $G$,
    then $M$ is said to be \emph{graphic}.
\end{defn}

We end this section by quickly defining a few more matroid-theoretic terms.
Let $M = (E,\mathcal{I})$ be a matroid.
Subsets of $E$ not in $\mathcal{I}$ are called \emph{dependent sets}.
Maximal elements of $\mathcal{I}$ are called \emph{bases}.
A \emph{circuit} of $M$ is a dependent set whose proper subsets are all independent.
The \emph{rank} of a subset $S \subseteq E$ is the maximum cardinality of an independent subset of $S$.
One denotes this as a function by $r_M : 2^E \rightarrow \mathbb{Z}$.
The \emph{rank of $M$} is $r_M(E)$ and is denoted $\rank(M)$.
Given $S \subseteq E$, the \emph{closure} of $S$, denoted $\cl_M(S)$, is the maximal superset of $S$ with the same rank,
and a \emph{flat} of $M$ is a subset of $E$ that is equal to its own closure. A \emph{spanning set} of $M$ is a subset $S\subseteq E$  that contains a basis of $M$.

\subsection{Forbidden minors and representability}

Given a matroid $M = (E,\mathcal{I})$ and $e \in E$,
one canonically defines two matroids $M \setminus e$ and $M / e$ on ground set $E \setminus \{e\}$
called the \emph{deletion} and \emph{contraction}.
If $\rank(\{e\}) \neq 0$, then independent sets are, respectively
\[
    \{I \in \mathcal{I} : e \notin I\} \qquad {\rm and} \qquad \{I \subseteq E: I \cup \{e\} \in \mathcal{I} \}.
\]
If $\rank(\{e\}) = 0$ then the independent sets of both the deletion and contraction are given by the first formula.
A \emph{minor} of $M$ is a matroid obtained from $M$ via a sequence of deletions and contractions.
The \emph{dual} $M^*$ of $M$ is the matroid on ground set $E$ whose independent sets are
\[
    \{I \subseteq E\setminus B : B {\rm \ is \ a \ basis \ of \ } M\}.
\]
Note that $M / e = (M^* \setminus e)^*$ for each $e \in E$.

All minors of a $\mathbb{K}$-representable matroid are $\mathbb{K}$-representable~\cite[Proposition 3.2.4]{oxley}.
Therefore, the class of $\mathbb{K}$-representable matroids can be classified via a list of minimal forbidden minors.
In other words, for each field $\mathbb{K}$, there exists a set $\mathcal{M}_\mathbb{K}$ of non-$\mathbb{K}$-representable matroids,
all of whose minors are $\mathbb{K}$-representable, such that a matroid $M$ is $\mathbb{K}$-representable if and only if it has no minors in $\mathcal{M}_\mathbb{K}$.
A well-known conjecture, often called \emph{Rota's conjecture},
states that $\mathcal{M}_{\mathbb{K}}$
is finite whenever $\mathbb{K}$ is finite~\cite{geelen2014solving}.
Rota's conjecture is known to be true for $\mathbb{F}_q$ for $q = 2,3,4$~\cite[Chapter~6.5]{oxley}.
We will later use the explicit forbidden minors for $\mathbb{F}_2$ and $\mathbb{F}_3$ representability
which we now state.

\begin{theorem}[{\cite[Theorems 6.5.4 and 6.5.7]{oxley}}]\label{theorem: forbidden minors finite field representability}
    Let $M$ be a matroid.
    Then $M$ is $\mathbb{F}_2$-representable if and only if $M$ has no minor isomorphic to $U_{2,4}$,
    and $M$ is $\mathbb{F}_3$-representable if and only if $M$ has no minor isomorphic to $U_{2,5},U_{3,5},F_7$ or~$F_7^*$.
\end{theorem}

Minors of graphic matroids are graphic~\cite[Corollary 3.2.2]{oxley},
so the class of graphic matroids can be defined via excluded minors as in Theorem~\ref{theorem: forbidden minors graphic} below.
Recall that $K_n$ denotes the complete graph on $n$ vertices and $K_{m,n}$ denotes the complete bipartite
graph on partite sets of size $m$ and $n$.

\begin{theorem}[{\cite[Theorem~10.3.1]{oxley}}]\label{theorem: forbidden minors graphic}
    A matroid is graphic if and only if it has no minor isomorphic to any of
    $U_{2,4},F_7,F_7^*,M(K_5)^*,$ or $M(K_{3,3})^*$.
\end{theorem}




\section{Flag Matroids}\label{section: flag matroids}
A \emph{flag} is a nested sequence of linear subspaces.
More formally, given a vector space $V$, a flag is a sequence $(L_1,\dots,L_k)$
of linear subspaces of $V$ so that $L_i \subseteq L_{i+1}$ for each $i$.
In this section, we formally define \emph{flag matroids}
which combinatorially abstract flags in the same way that matroids combinatorially abstract linear subspaces.
Our first order of business is to define flag matroids in their usual axioms,
then provide an alternative, but equivalent, set of axioms.

\begin{definition}
    Let $M$ and $N$ be matroids on the same ground set $E$. $N$ is said to be a \emph{lift} of $M$, or equivalently $M$ is a \emph{quotient} of $N$, if every flat of $M$ is a flat of $N$.
\end{definition}

\begin{prop}[{\cite[Proposition 7.4.7]{white1986theory}}, {\cite[Proposition 7.3.6]{oxley}}]\label{prop: equivalent characterizations of lifts}
    The following are equivalent for matroids $M$ and $N$ on a common ground set $E$:
    \begin{enumerate}
        \item $N$ is a lift of $M$
        \item $M^*$ is a lift of $N^*$
        \item there exists a matroid $Q$ on ground set $E(Q)$ and some $X \subseteq E(Q)$ such that $M = Q / X$ and $N = Q \setminus X$
        \item if $X \subseteq E$, then $\cl_N(X) \subseteq \cl_M(X)$
        \item for each basis $B$ of $N$ and $e \in E\setminus B$, there exists some basis $B'$ of $M$ such that $B' \subseteq B$ and
        \begin{equation*}
            \{f : (B'\cup e)\setminus f \text{ is a basis of $M$ }\} \subseteq \{f : (B\cup e)\setminus f \text{ is a basis of $N$ }\}
        \end{equation*}
    \end{enumerate}
\end{prop}

We remark that if $N$ is a lift of $M$, the rank of $N$ is bounded below by the rank of $M$. If the ranks are equal, then $N = M$ and the lift is said to be \emph{trivial}. When the rank of $N$ is one greater than that of $M$, then the lift is said to be \emph{elementary}.

\begin{defn}\label{defn: lifts and flags}
    A \emph{flag matroid} is a sequence of matroids $(M_1,\dots,M_k)$ such that for each $i$, $M_{i+1}$ is a nontrivial lift of $M_i$.
\end{defn}

Taking inspiration from the way matroids can be equivalently defined in many ways,
we now offer the following alternative axiom system for flag matroids.
Theorem~\ref{thm: flag matroid cryptomorphism} establishes that Definitions~\ref{defn: lifts and flags} and~\ref{defn: flag matroid cryptomorphism} indeed define the same object. 

\begin{definition}\label{defn: flag matroid cryptomorphism}
     Let $E$ be a finite set and let $\mathcal{F}$ be a collection of non-empty subsets of $E$.  A pair $(E, \mathcal{F})$ is called a \emph{set-system flag matroid} if
     \begin{enumerate}
         \item if $F,G \in \mathcal{F}$ satisfy $|F| = |G|$ and $x \in F \setminus G$,
        then there exists $y \in G \setminus F$ such that $G \cup \{x\} \setminus \{y\} \in \mathcal{F}$ 
        \item if there exist sets in $\mathcal{F}$ of different cardinality, then for any $F \in \mathcal{F}$ of non-minimal cardinality and $e \in E\setminus F$, there exists some $G \in \mathcal{F}$ such that $G \subsetneq F$, $|G| = \max\{|S| : S\in \mathcal{F} \text{ and } |S| < |F|\}$, and
        \begin{equation*}
            \{f : (G\cup e)\setminus f \in \mathcal{F}\} \subseteq \{f : (F\cup e)\setminus f \in \mathcal{F}\}
        \end{equation*}
     \end{enumerate}
     We refer to $E$ as the \emph{ground set} and $\mathcal{F}$ as the \emph{feasible sets}.
     The \emph{rank} of a flag matroid is the size of its largest feasible set.
\end{definition}

The terminology ``feasible sets'' comes from the theory of \emph{greedoids}, a generalization of matroids. We will later see that certain classes of greedoids are flag matroids. See~\cite{greedoids} for more about greedoids.
Theorem~\ref{thm: flag matroid cryptomorphism} below tells us how
Definitions~\ref{defn: lifts and flags} and~\ref{defn: flag matroid cryptomorphism} describe the same object.

\begin{theorem}\label{thm: flag matroid cryptomorphism}
    Let $E$ be a finite set, let $\mathcal{F}$ be a set of subsets of $E$, and let
    let $j_1 < \dots < j_k$ be the cardinalities of elements of $\mathcal{F}$.
    For $i = 1,\dots,k$ define
    \[
        \mathcal{B}_i := \{F \in \mathcal{F} : |F| = j_i\}.
    \]
    Then $(E,\mathcal{F})$ is a set-system flag matroid if and only if each $\mathcal{B}_i$ is the set of bases of a matroid $M_i$
    and $(M_1,\dots,M_k)$ is a flag matroid.
\end{theorem}
\begin{proof}
    First assume $(E,\mathcal{F})$ is a flag matroid.
    The first condition in Definition~\ref{defn: flag matroid cryptomorphism} implies that each $\mathcal{B}_i$
    is the set of bases of a matroid $M_i$.
    The second condition, together with Proposition~\ref{prop: equivalent characterizations of lifts},
    implies that each $M_{i+1}$ is a lift of $M_i$.

    Now assume each $\mathcal{B}_i$ is the set of bases of a matroid $M_i$ such that $(M_1,\dots,M_k)$ is a flag matroid.
    If $F,G \in \mathcal{F}$ have the same cardinality then $F,G \in \mathcal{B}_i$ for some $i$.
    By the basis exchange axiom, $(E,\mathcal{F})$ satisfies the first condition of Definition~\ref{defn: flag matroid cryptomorphism}.
    Since each $M_{i+1}$ is a lift of $M_i$, Proposition~\ref{prop: equivalent characterizations of lifts} implies the second condition is satisfied as well.
\end{proof}

We call the sequence $(M_1, \dots, M_r)$ of matroid lifts associated to a set-system flag matroid $\mathfrak{F} = (E,\mathcal{F})$ by Theorem~\ref{thm: flag matroid cryptomorphism}
the \emph{sequential representation} of $(E,\mathcal{F})$.
The following definition gives a family of flag matroids associated to each matroid.

\begin{defn}\label{defn: flag matroids from matroid}
    Let $M$ be a matroid on ground set $E$ with independent sets $\mathcal{I}$ and spanning sets $\mathcal{S}$.
    For each pair of integers $0 \le s \le r \le |E|$, define
    \begin{align*}
        \mathfrak{F}_{s,r}(M) &:= (E,\{S \in \mathcal{I} \cup \mathcal{S}: s \le |S| \le r\}) \qquad
        \mathcal{I}(M):=\mathfrak{F}_{0,r_M(E)} \\
        &\mathcal{B}(M) := \mathfrak{F}_{r_M(E),r_M(E)} \qquad
        \mathcal{S}(M):=\mathfrak{F}_{r_M(E),|E|}.
    \end{align*}
\end{defn}
We claim that $\mathfrak{F}_{s,r}(M)$ is a set-system flag matroid.
The first axiom follows from the matroid basis exchange axiom.
Indeed, the independent sets of $M$ of size $k \le r$ are the bases
of the $(r-k)^{\rm th}$ truncation of $M$ and the spanning sets of size $k \ge r$ are the bases of the $(k-r)^{\rm th}$ elongation.
The second axiom follows the hereditary axiom for independent sets of a matroid if $S$ is independent, and it follows from the fact that supersets of spanning sets are spanning if $S$ is a spanning set.

\begin{remark}\label{remark: flag matroid axioms genearlize matroid axioms}
    Since $\mathcal{I}(M), \mathcal{B}(M)$ and $\mathcal{S}(M)$ are all set-system flag matroids,
    one can think of the set-system flag matroid axioms given in Definition~\ref{defn: flag matroid cryptomorphism} as a simultaneous generalization of the independent set, basis, and spanning set axioms of a matroid.    
\end{remark}

Set-system flag matroids that are not matroids that have a feasible set of every size between $0$ and some $r \ge 0$ have been studied
previously under the names \emph{Gaussian elimination greedoids}~\cite{grinberga2020bhargava,grinberg2021greedoid} and \emph{Gauss greedoids}~\cite{greedoids}.

One can make an analogy between feasible sets of a set-system flag matroid and independent sets of a matroid and define the \emph{rank} of a subset of the ground set of a flag matroid to be the maximum cardinality of a feasible subset.
Of course, the feasible sets of a set-system flag matroid can be recovered from its rank function and vice versa,
but it would be interesting to know if there were a characterization of set-system flag matroid rank functions similar to the cryptomorphic definition of matroids in terms of their rank functions.
Since subsets of feasible sets of set-system flag matroids need not be feasible,
a satisfactory notion of circuits for flag matroids is not so obvious
and an axiom system in terms of them, even less so.
Answering these questions would make for interesting future research.

For the rest of this paper, we will be primarily working with the language of set-system flag matroids but we will simply just refer to them as flag matroids.

\subsection{Representable Flag Matroids}

Given a matrix $A \in \mathbb{K}^{r\times n}$
and $d \le r$, we let $A_{\le d}$ denote the $d\times n$ submatrix of $A$ obtained by restricting to the first $d$ rows.
Given $F \subseteq \{1,\dots,n\}$, we let $A^F$ denote the submatrix of $A$ consisting of the columns indexed by $F$.

\begin{defn}\label{defn: flag matroid from matrix}
    Let $\mathbb{K}$ be a field, suppose $A \in \mathbb{K}^{r\times n}$ and let $1 \le d_1 < \dots < d_k =~r$ be integers.
    Define $\mathfrak{F}(A; d_1, \dots, d_k)$ to be the flag matroid with sequential representation
    \[
        (M(A_{\le d_1}),\dots,M(A_{\le d_k}))
    \]
    When $k=r$ and $d_i = i$ for all $i$, we denote this as $\mathfrak{F}(A)$.
\end{defn}

Flag matroids expressible as $\mathfrak{F}(A; d_1,\dots,d_k)$ for some matrix $A$ with entries in $\mathbb{K}$ are called \emph{$\mathbb{K}$-representable}.
The feasible sets of $\mathfrak{F}(A;d_1,\dots,d_k)$ are the subsets $F \subseteq \{1,\dots,n\}$ with cardinality equal to some $d_i$
such that $A_{\leq |F|}^F$ is nonsingular.
If $(L_1,\dots,L_k)$ is a flag in a $\mathbb{K}$-vector space $V$ with basis $E$
and $A$ is any matrix whose first $\dim(L_i)$ rows are a basis of $L_i$ in $E$ coordinates,
then $(M(L_1),\dots,M(L_k))$ is the sequential representation of $\mathfrak{F}(A;\dim(L_1),\dots,\dim(L_k))$.

\begin{ex}
    The flag matroid $\mathcal{I}(U_{2,3})$ is $\mathbb{K}$-representable for any $\mathbb{K}$ with cardinality $3$ or greater.
    Indeed, let $x \in \mathbb{K} \setminus \{0,1\}$ and define
    \begin{equation*}
        A:= \begin{pmatrix}
            1 & 1 & 1 \\
            0 & 1 & x
        \end{pmatrix}.
    \end{equation*}
    Then $(A;1,2)$ is a $\mathbb{K}$-representation of $\mathcal{I}(U_{2,3})$.
    However, $\mathcal{I}(U_{2,3})$ is not $\mathbb{F}_2$-representable.
    Indeed, since all singleton sets are feasible, the first row of any representation of $\mathcal{I}(U_{2,3})$
    can have no zeros, and since all two-element sets are independent, the second row can have no repeated entries.
\end{ex}

We call flag matroids of the form $\mathcal{I}(U_{r,n})$ \emph{uniform}.
The following Theorem characterizes representability of uniform flag matroids.

\begin{thm}\label{thm: uniform flag matroids representability}
    The flag matroid $\mathcal{I}(U_{r,n})$ is $\mathbb{K}$-representable if and only if $r\leq 1$ or $|\mathbb{K}| \geq n$.
\end{thm}
\begin{proof}
    The constant matrix of size $1\times n$, with each entry equal to $0$ (respectively $1$) is a representation of $U_{0,n}$ (respectively $U_{1,n}$) over any field.
    Suppose now $|\mathbb{K}| \geq n$.
    Let $e_1, e_2, \dots, e_n$ be distinct elements of $\mathbb{K}$ and define 
    \begin{equation*}
        A := \begin{pmatrix}
            1 & 1 & \cdots & 1 \\
            e_1 & e_2 & \cdots & e_n \\
            e_1^2 & e_2^2 & \cdots & e_n^2 \\
            \vdots & \vdots & \ddots & \vdots \\
            e_1^{r-1} & e_2^{r-1} & \dots & e_n^{r-1}
        \end{pmatrix}
    \end{equation*}
    If $F \subseteq \{1,\dots,n\}$ with $|F| \le r$,
    then $A_{1, 2, \dots, |F|}^F$ is a Vandermonde matrix and therefore nonsingular.
    Since $A$ has size $r\times n$, this implies that $\mathfrak{F}(A) = \mathcal{I}(U_{r,n})$.

    Now assume $\mathcal{I}(U_{r,n})$ is $\mathbb{K}$-representable and suppose $r > 1$.
    The first row of $A$ cannot contain any zero entries
    so by appropriately scaling each column we may assume that the first row of $A$ is all ones.
    Since all subsets of size 2 are feasible in $\mathcal{I}(U_{r,n})$,
    all entries of the second row of $A$ must be distinct and so $|\mathbb{K}| \ge n$.
\end{proof}

\subsection{Graphic Flag Matroids}
We now discuss a class of flag matroids analogous to the class of graphic matroids.
This class of matroids was discussed in greater generality using different language in~\cite[Section~7.4]{white1986theory}.
We begin by recalling some basic terminology regarding set partitions that will be necessary.

Given a set $S$, a \emph{partition} of $S$ is a set of disjoint subsets of $S$ whose union is $S$.
The sets in a partition are called \emph{cells}.
There is a partial order on the set of partitions of a set $S$.
In particular, given two partitions $\mathcal{P} = \{P_1,\dots,P_k\}$ and $\mathcal{Q} = \{Q_1,\dots,Q_l\}$ of $S$,
one says that $\mathcal{Q}$ is \emph{finer} than $\mathcal{P}$ (equivalently, $\mathcal{P}$ is \emph{coarser} than $\mathcal{Q}$)
if each $Q_i$ is a subset of some $P_j$.
We express this symbolically as $\mathcal{P} \succ \mathcal{Q}$.
A \emph{chain of partitions} is a sequence $\mathcal{P}_1,\dots,\mathcal{P}_k$ of partitions of a set $S$
such that $\mathcal{P}_i \succ \mathcal{P}_{i+1}$ for each $i$.

\begin{defn}\label{defn: graphicl flag matroids}
    Let $G = (V,E)$ be a graph and let $\mathcal{P} = \{P_1,\dots,P_k\}$ be a partition of $V$.
    For each $e = uv \in E$ define $e^P := P_i P_j$ where $u \in P_i$ and $v \in P_j$.
    For each $S \subseteq E$ define $S^P := \{e^P : e \in S\}$
    and define $M(G,\mathcal{P})$ to be the matroid on ground set $E$ where $S \subseteq E$ is independent
    if and only if the graph $(\mathcal{P},S^\mathcal{P})$ has no cycles.
\end{defn}

\begin{prop}\label{prop: graphic flag matroids are flag matroids}
    If $G = (V,E)$ is a graph and $\mathcal{P}_1 \succ \dots \succ \mathcal{P}_k$ is a chain of partitions of $V$,
    then $(M(G,\mathcal{P}_1),\dots,M(G,\mathcal{P}_k))$ is the sequential representation of a flag matroid.
    If $G$ is not connected, then there exists a connected graph $H$ and a sequence of partitions $(\mathcal{Q}_1,\dots,\mathcal{Q}_k)$ of the vertices of $H$ such that
    $M(G,\mathcal{P}_i) = M(H,\mathcal{Q}_i)$ for each $i$.
\end{prop}
\begin{proof}
    Let $\cl_i$ and $\cl_{i+1}$ be the closure operators of $M(G, \mathcal{P}_i)$ and $M(G, \mathcal{P}_{i+1})$ respectively. In light of Proposition~\ref{prop: equivalent characterizations of lifts},
    it suffices to show that for $S \subseteq E$, $\cl_{i+1}(S) \subseteq \cl_i(S)$. So let $e \in \cl_{i+1}(S)$.
    Then the graph $(\mathcal{P}_{i+1},S^{\mathcal{P}_{i+1}}_{i+1})$ has a path $L = (e_1^{\mathcal{P}_{i+1}},\dots,e_l^{\mathcal{P}_{i+1}})$ connecting the endpoints of $e^{\mathcal{P}_{i+1}}$.
    Then $e \in \cl_{i}(S)$ because the path $e_1^{\mathcal{P}_{i}},\dots,e_l^{\mathcal{P}_i}$ in $(\mathcal{P}_{i},S^{\mathcal{P}_i}_{i})$
    connects the endpoints of $e^{\mathcal{P}_{i}}$ as it is obtained from $L$ by contracting edges.

    If $G$ is disconnected, let $H$ be a graph obtained from $G$ by choosing a vertex in each connected component of $G$ and identifying them together in a single vertex.
    Let $\mathcal{Q}_i$ be the partition of the vertices of $H$ obtained from $\mathcal{P}_i$
    by making the corresponding identifications.
    Let $E$ denote the edge set of $G$.
    Then each graph $(\mathcal{Q}_i,E^{\mathcal{Q}_i})$ is obtained from $(\mathcal{P}_i,E^{\mathcal{P}_i})$ by identifying vertices from different connected components.
    Then $M(G,\mathcal{P}_i) = M(H,\mathcal{Q}_i)$ for each $i$ because identifying vertices from different connected components does not change the matroid of a graph.
\end{proof}

\begin{defn}
    The class of flag matroids that can arise as in Proposition~\ref{prop: graphic flag matroids are flag matroids} are called \emph{graphic flag matroids}.
\end{defn}

\begin{ex}\label{ex: graphic flag matroid}
Let $G = K_4$, where we denote the vertices as $V = \{1, 2, 3, 4\}$ and edges as $E = \{a, b, c, d, e, f\}$. Define $\mathcal{P}_1 = 1234$, $\mathcal{P}_2 = 124|3$, $\mathcal{P}_3 = 12|3|4$, and $\mathcal{P}_4 = 1|2|3|4$. We can graphically represent the sequential representation of this flag matroid via the sequence of graphs shown in Figure~\ref{figure: graphic flag matroid example}.

\begin{figure}[H]
    \centering
\resizebox{\columnwidth}{!}
{
    \subfloat[$(\mathcal{P}_1, S^{\mathcal{P}_1})$]{
        \begin{tikzpicture}
            \vertex[fill] (1234) at (0,1) {};
  \node[]  at (1,0) {$1234$};
  \path
  (1234) edge[in = 60, out = 0, loop, looseness = 40] node[above] {$c$} (1234)
  (1234) edge[in = 120, out = 60, loop, looseness = 40] node[above] {$e$} (1234)
  (1234) edge[in = 180, out = 120, loop, looseness = 40] node[above] {$b$} (1234)
  (1234) edge[in = 240, out = 180, loop, looseness = 40] node[below] {$a$} (1234)
  (1234) edge[in = 300, out = 240, loop, looseness = 40] node[below] {$d$} (1234)
  (1234) edge[in = 360, out = 300, loop, looseness = 40] node[right, pos = 0.7] {$f$}(1234)
  ;
        \end{tikzpicture}
    }
    \subfloat[$(\mathcal{P}_2, S^{\mathcal{P}_2})$]{
    \begin{tikzpicture}
            \vertex[fill] (124) at (0,1) {};
  \vertex[fill] (3) at (2,1) [label=above:$3$] {};
  \node[] at (0.60, 0) {$124$};
  \path
  (124) edge node[above, pos=0.4] {$f$} (3)
  (124) edge[bend left = 70] node[above, pos=0.7] {$c$} (3)
  (124) edge[bend right = 70] node[below, pos=0.8] {$d$} (3)
  (124) edge[in = 120, out=60, loop, looseness = 40] node[above] {$e$} (124)
  (124) edge[in = 210, out = 150, loop, looseness = 40] node[left] {$b$} (124)
  (124) edge[in = 240, out=300, loop, looseness = 40] node[below] {$a$} (124)
  ;
        \end{tikzpicture}
    }
    \subfloat[$(\mathcal{P}_3, S^{\mathcal{P}_3})$]{
        \begin{tikzpicture}
             \vertex[fill] (12) at (0,1) {};
  \vertex[fill] (4) at (2,2) [label=above:$4$] {};
  \vertex[fill] (3) at (2,0) [label=below:$3$] {};
  \node[] at (0, 0.4) {$12$};
\path 
(12) edge node[below] {$e$} (4)
(12) edge node[above] {$f$} (3)
(12) edge[out=220, loop, looseness = 34] node[above, pos = 0.75]{$a$} (12)
(4) edge node[right] {$c$} (3)
(12) edge[bend left] node[above] {$b$} (4)
(12) edge[bend right] node[below] {$d$} (3)
;
        \end{tikzpicture}
    }
    \subfloat[$(\mathcal{P}_4, S^{\mathcal{P}_4})$]{
        \begin{tikzpicture}
            \vertex[fill] (2) at (8,0) [label=below:$2$] {};
  \vertex[fill] (3) at (10,0) [label=below:$3$] {};
  \vertex[fill] (4) at (10,2) [label=above:$4$] {};
  \vertex[fill] (1) at (8,2) [label=above:$1$] {};
\path
(1) edge node[left] {$a$} (2) 
(1) edge node[above] {$b$} (4)
(1) edge node[above, pos=0.40] {$f$} (3)
(2) edge node[below, pos=0.33] {$e$} (4)
(2) edge node[below] {$d$} (3)
(4) edge node[right] {$c$} (3)
;
        \end{tikzpicture}
    }
}
\caption{Graphic representation of $\mathfrak{F}(K_4, \mathcal{P}_1, \mathcal{P}_2, \mathcal{P}_3, \mathcal{P}_4)$}\label{figure: graphic flag matroid example}
\end{figure}
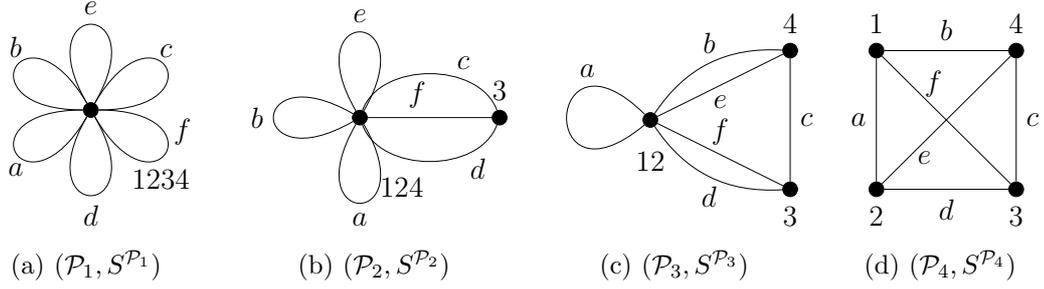

\end{ex}

\subsection{Minors and duality}
We now discuss minors and duality for flag matroids.
Such concepts have already been defined for flag matroids as sequences of matroid lifts
in~\cite{brandt2021tropical} and~\cite{garcia2018exploring}.
Our contribution in this section is to extend this same minor and duality theory to our view of a flag matroid as a set system.

Our motivation for studying flag matroid minors is to provide a framework that can be used
to characterize certain classes of flag matroids.
We will show that all minors of graphic flag matroids are graphic,
and that all minors of $\mathbb{K}$-representable flag matroids are $\mathbb{K}$-representable.
In principle, this allows us to characterize each such class in terms of forbidden minors,
and this is something we explore for $\mathbb{K}$-representable flag matroids in the next section.
For now, we begin by defining flag matroid minors and duality.

\begin{defn}\label{defn: flag matroid minors and duality}
    Let $\mathfrak{F} = (E,\mathcal{F})$ be a flag matroid and $e \in E$.
    Then, define
    \begin{enumerate}
        \item the \emph{deletion} of $\mathfrak{F}$ by $e$ to be \ \ \ $\mathfrak{F} \setminus e := (E\setminus \{e\},\{F \in \mathcal{F} : e \notin F\})$
        \item the \emph{dual} of $\mathfrak{F}$ to be \ \ \ $\mathfrak{F}^* := (E,\{E \setminus F : F \in \mathcal{F}\})$
        \item the \emph{contraction} of $\mathfrak{F}$ by $e$ to be \ \ \ $\mathfrak{F} / e := (\mathfrak{F}^*\setminus e)^*$
        \item the \emph{ith chopping} of $\mathfrak{F}$ to be \ \ \ $C_{-i}(\mathfrak{F}) := (E,\{F \in \mathcal{F} : |F| \neq i\})$
    \end{enumerate}
    A flag matroid obtained from $\mathfrak{F}$ via a sequence of deletion, contraction, and chopping operations is called a \emph{minor} of $\mathfrak{F}$.
\end{defn}

The next proposition tells us that taking minors and duals of flag matroids corresponds to analogous operations on the constituent matroids of the sequential representation.

\begin{prop}\label{prop: minors sequential representation}
    Let $\mathfrak{F} = (E, \mathcal{F})$ be a flag matroid and let $e \in E$. Suppose $\mathfrak{F}$ has as sequential representation $\mathcal{M} = (M_1, M_2, \dots, M_r)$. Then 
    \begin{enumerate}
        \item $\mathfrak{F}\setminus e$ has as sequential representation $(M_1\setminus e, M_2\setminus e, \dots, M_r\setminus e)$ 
        \item $\mathfrak{F}^*$ has as sequential representation $ (M_{r}^*, M_{r-1}^*, \dots, M_1^*)$ 
        \item $\mathfrak{F}/e$ has as sequential representation $(M_1 /e, M_2/e, \dots, M_r/e)$
        \item $C_{-i}(\mathfrak{F})$ has as sequential representation $(M_1, M_2, \dots, M_{i-1}, M_{i+1}, \dots, M_r)$
    \end{enumerate}
\end{prop}

\begin{proof}
    The fourth claim follows immediately from the definition of chopping.
    The third claim follows immediately from the first two.
    
    For the first claim, note that $F$ is feasible in $\mathfrak{F}\setminus e$ if and only if $F$ is feasible in $\mathfrak{F}$ and $F \subseteq E\setminus e$, if and only if $F$ is a basis of $M_{|F|}$ in $(M_1, M_2, \dots, M_r)$ and $F\subseteq E\setminus e$, if and only if $F$ is a basis of $M_i\setminus e$ in $(M_1\setminus e, M_2\setminus e, \dots, M_r\setminus e)$.

    For the second claim, note that $F$ is feasible in $\mathfrak{F}^*$ if and only if $E\setminus F$ is feasible in $\mathfrak{F}$, if and only if $E\setminus F$ is a basis of $M_{|E\setminus F|}$ in $(M_1, M_2, \dots, M_r)$, if and only if $F$ is a basis of $M_{|E\setminus F|}^*$ in $(M_r^*, M_{r-1}^*, \dots, M_1^*)$.
\end{proof}

In light of Theorem~\ref{thm: flag matroid cryptomorphism}, Proposition~\ref{prop: minors sequential representation} tells us that the class of flag matroids is closed under taking minors, i.e.~that a minor of a flag matroid is again a flag matroid. 
We now state properties of minors of flag matroids that mimic properties of minors for matroids.

\begin{prop}\label{prop: minors commute}
    Let $\mathfrak{F} = (E, \mathcal{F})$ be a flag matroid and let $X, Y \subseteq E$ such that $X$ and $Y$ are disjoint. Then
    \begin{enumerate}
        \item $\mathfrak{F}\setminus X \setminus Y = \mathfrak{F}\setminus (X \cup Y) = \mathfrak{F}\setminus Y \setminus X$
        \item $\mathfrak{F}/ X / Y = \mathfrak{F}/ (X \cup Y) = \mathfrak{F}/ Y /X$
        \item $\mathfrak{F}\setminus X / Y = \mathfrak{F}/ Y \setminus X$
        \item $(\mathfrak{F}^*)^* = \mathfrak{F}$
    \end{enumerate}
\end{prop}
\begin{proof}
    These immediately follow from Theorem~\ref{thm: flag matroid cryptomorphism} and Proposition~\ref{prop: minors sequential representation}.
\end{proof}

In light of Proposition~\ref{prop: minors commute}, we can write any minor of $\mathfrak{F}$ as $\mathfrak{F}/X\setminus Y$, where $X, Y \subseteq E$ and are disjoint. 

Similar to the matroid setting, $\mathbb{K}$-representable flag matroids are closed under taking minors and duals.
This was shown for a generalization of flag matroids in~\cite{jarra2024flag},
but is instructive to have a proof for the special case of flag matroids which we  now provide.

\begin{theorem}[{c.f.~\cite[Theorems 2.13 and 2.15]{jarra2024flag}}]\label{thm: minors and duals of representable}
    Minors and duals of $\mathbb{K}$-representable flag matroids are $\mathbb{K}$-representable.
\end{theorem}
\begin{proof}
    Let $\mathfrak{F}$ be an arbitrary $\mathbb{K}$-representable flag matroid, then there is some $A\in \mathbb{K}^{n\times r}$ and sequence of integers $1 \le d_1 < \dots < d_k = r$ so that $\mathfrak{F} := \mathfrak{F}(A;d_1,\dots,d_r)$. 
    Any chopping of $\mathfrak{F}$ is of the form $\mathfrak{F}(A;d_{i_1},\dots,d_{i_k})$ and therefore $\mathbb{K}$-representable.
    Any deletion of $\mathfrak{F}$ is of the form $\mathfrak{F}(B;d_1,\dots,d_r)$ where $B$ is obtained from $A$ by removing a column.
    
    By definition of contraction, it now suffices to show $\mathfrak{F}^*$ is $\mathbb{K}$-representable.
    Since the kernel of each $A_{\le d_i}$ lies in the kernel of $A_{\le d_{i-1}}$,
    we can choose $x_1,\dots,x_{n-d_1} \in \mathbb{K}^n$ such that
    $x_1,\dots,x_{n-d_i}$ is a basis for the kernel of $A_{\le d_i}$ for each $i$.
    Let $B$ denote the matrix whose rows are $x_1,\dots,x_{n-d_1}$.
    Then $A_{\le d_i} (B_{\le n-d_i})^T = 0$.
    It then follows from \cite[Theorem 2.2.8]{oxley}
    that $M(A_{\le d_i})^* = M(B_{\le n-d_i})$ for each $i$.
    Proposition~\ref{prop: minors sequential representation} then implies that $\mathfrak{F}^* = \mathfrak{F}(B;n-d_k,\dots,n-d_1)$.
\end{proof}

\begin{theorem}\label{thm: minors of graphic}
    Every minor of a graphic flag matroid is graphic. 
\end{theorem}
\begin{proof}
    Let $\mathfrak{F}$ be a graphic flag matroid. Then there is a graph $G = (V,E)$ and a chain of partitions $\mathcal{P}_1 \succ \dots \succ \mathcal{P}_r$ of $V$ such that the sequential representation of $\mathfrak{F}$ is $(M(G, \mathcal{P}_1), \dots, M(G, \mathcal{P}_r))$.
    Let $e \in E$ and let $D$ and $T$ denote the graphs obtained by, respectively, deleting and contracting $e$ in $G$.
    We then have the following sequential representations of $\mathfrak{F} \setminus e$ and $\mathfrak{F} / e$
    \[
        (M(D, \mathcal{P}_1), \dots, M(D, \mathcal{P}_k)) \qquad (M(T, \mathcal{Q}_1), \dots, M(T, \mathcal{Q}_k))
    \]
    where $\mathcal{Q}_i$ is the partition of $V$ obtained from $\mathcal{P}_i$ by taking the union of the parts containing the endpoints of $e$
    and replacing each such vertex with the new vertex of $T$.
    The $i$th chopping $C_{-i}(\mathfrak{F})$ has the following sequential representation
    \[
        (M(G, \mathcal{P}_1), \dots, M(G, \mathcal{P}_{i-1}), M(G, \mathcal{P}_{i+1}), \dots, M(G, \mathcal{P}_r)). \qedhere
    \]
\end{proof}

Just as with matroids, the dual of graphic flag matroid need not be graphic.
In particular, given a graph $G$, the dual graphic matroid $M(G)^*$ is graphic if and only if $G$ is planar~\cite{oxley}.
Therefore for any graph $G$, the dual flag matroid of $\mathcal{B}(M(G))$ is a graphic if and only if $G$ is planar.

\subsection{Majors of Flag Matroids}
Given any flag matroid $\mathfrak{F} = (E, \mathcal{F})$, it can be shown there exists a matroid $Q$ such that every collection of feasible sets of cardinality $i$ are precisely the bases of some minor of $Q$.
Equivalently, every matroid in the sequential representation of $\mathfrak{F}$ can be written as a minor of $Q$.
Such a $Q$ is known as a \emph{major} of $\mathfrak{F}$~\cite{kung1986strong}.

\begin{defn}[\cite{kung1986strong}]\label{defn: major}
    Let $\mathfrak{F}$ be a flag matroid on ground set $E$ with sequential representation $(M_1,\dots,M_k)$.
    A matroid $Q$ on ground set $E \cup X$ is a \emph{major} of $\mathfrak{F}$
    if $X$ is independent in $Q$ and there exists an ordered partition
    $X = X_1 \cup \dots \cup X_{k-1}$ such that the following holds for $i = 1,\dots,k$
    \[
        M_i = Q / (X_1 \cup \dots \cup X_{i-1}) \setminus (X_i \cup \dots \cup X_{k-1}).
    \]
\end{defn}

When working with a major of a flag matroid, it will be easier to work with
$I_i:=X_1 \cup \dots \cup X_{i-1}$ and $J_i:=X_i \cup \dots \cup X_{k-1}$ instead of the $X_i$'s directly.

\begin{ex}\label{ex: major example}
    Let $\mathfrak{F}$ be the flag matroid on $\{e_1,e_2,e_3\}$ such that every subset is feasible.
    The sequential representation of $\mathfrak{F}$ is $(U_{1,3}, U_{2,3}, U_{3,3})$.
    Then $U_{3,5}$ is a major of $\mathfrak{F}$.
    Indeed, if $\{e_1,\dots,e_5\}$ is the ground set of $U_{3,5}$, then
    \[
        U_{1,3} = U_{3,5} /\{e_4,e_5\} \qquad U_{1,3} = U_{3,5} \setminus e_4 / e_5 \qquad U_{3,3} = U_{3,5} \setminus \{e_4,e_5\}
    \]
\end{ex}

\begin{thm}\label{thm: majors exist}
    Every flag matroid has a major.
\end{thm}
\begin{proof}
    See~\cite[Section 8.2]{kung1986strong}.
\end{proof}

Majors of flag matroids need not be unique.
Indeed, Example~\ref{ex: major example} implies that $U_{3,5}$ is a major of $(U_{1,3}, U_{3,3})$.
However, the linear matroids of the following matrices are also majors.
They are not isomorphic to each other nor to $U_{3,5}$.
\begin{equation*}
    \begin{pmatrix}
        1 & 1 & 1 & 0 & 0 \\
        0 & 1 & 1 & 1 & 0 \\
        0 & 0 & 1 & 0 & 1
    \end{pmatrix} \text{ over $\mathbb{F}_2$} \quad \text{and} \quad 
    \begin{pmatrix}
        1 & 1 & 1 & 0 & 0 \\
        0 & 1 & 2 & 1 & 0 \\
        0 & 1 & 1 & 0 & 1 \\
    \end{pmatrix} \text{ over $\mathbb{F}_3$}
\end{equation*}

It was shown in~\cite[Lemma~2.1]{mundhe2019graphic} that a flag matroid of form $(M, N)$ is $\mathbb{F}_2$-representable if and only if it has a major that is $\mathbb{F}_2$-representable.
More generally, for any flag matroid $\mathbb{F}$ and any field $\mathbb{K}$, existence of a $\mathbb{K}$-representable major of $\mathfrak{F}$
is equivalent to $\mathbb{K}$-representability of $\mathfrak{F}$.
As with Theorem~\ref{thm: minors and duals of representable}, this was shown for a generalization of flag matroids in~\cite{jarra2024flag},
but is instructive to have a proof for the special case of flag matroids which we now provide.

\begin{theorem}[{\cite[Theorem 2.23 and Remark 2.24]{jarra2024flag}}]\label{thm: representable iff has representable major}
    A flag matroid is $\mathbb{K}$-representable if and only if it has a $\mathbb{K}$-representable major.
\end{theorem}
\begin{proof}
    Let $\mathfrak{F}$ be a $\mathbb{K}$-representable flag matroid with sequential representation $(M_1, \dots, M_k)$. Let $A$ be its $\mathbb{K}$-representation.
    Now, consider the new matrix obtained by adding an identity matrix of size $s\times s$, where $s = \text{rank}(M_k)-\text{rank}(M_1)$, to the bottom right corner and having a zero matrix in the upper right corner
    \begin{equation*}
        A' := \begin{pmatrix}
            \begin{array}{c|c}
                \multirow{3}{*}{\scalebox{1.5}{$A$}} & \mathlarger{\textbf{0}} \\
                 & \\
                & \textbf{$I_s$}
            \end{array}
        \end{pmatrix}
    \end{equation*}
    We now show that the linear matroid of $A'$, which we denote as $Q$, is a major of $\mathfrak{F}$.
    Let $E' := \{e_1, \dots, e_s\}$ be the set of columns on the right side of $A'$.
    Then $M_1=Q/E'$ and $M_k = Q\setminus E'$.
    For an arbitrary $M_i$, we have that $M_i = Q/K_i\setminus L_i$
    where $L_i = \{e_1,\dots,e_{d_i}\}$ and $K_i = \{e_{d_i+1},\dots,e_{d_k}\}$.

    Now suppose $Q$ is a $\mathbb{K}$-representable major for the flag matroid $\mathbb{K}$. We have that $M_1 = Q/X$ and $M_k = Q\setminus X$ for some $X \subseteq E(Q)$. Assuming $X$ is independent, the bottom-right $|X| \times |X|$ submatrix will be nonsingular, and therefore via row operations can be turned into an identity matrix of size $s \times s$, were $s = |X| = \text{rank}(M_k) - \text{rank}(M_1)$. We can then turn the upper-right corner above the identity matrix into the zero matrix then. As a result, we have a $\mathbb{K}$-representation of $Q$ that is of form
    \begin{equation*}
        D := \begin{pmatrix}
            \begin{array}{c|c}
                \multirow{3}{*}{\scalebox{1.5}{$C$}} & \mathlarger{\textbf{0}} \\
                 & \\
                & \textbf{$I_s$}
            \end{array}
        \end{pmatrix}
    \end{equation*}
    The matrix $C$ will be the $\mathbb{K}$-representation of the flag matroid. 
\end{proof}

The matroid of the graph in Figure~\ref{fig: graphic major}
is a major of the graphic flag matroid from Example~\ref{ex: graphic flag matroid}.
In fact, every graphic flag matroid has a graphic major.

\begin{figure}
    \centering
    \begin{tikzcd}
   & \begin{tikzpicture}
            \vertex[fill] (2) at (8,0) [label=below:$2$] {};
  \vertex[fill] (3) at (10,0) [label=below:$3$] {};
  \vertex[fill] (4) at (10,2) [label=above:$4$] {};
  \vertex[fill] (1) at (8,2) [label=above:$1$] {};
\path
(1) edge node[left] {} (2) 
(1) edge node[below] {} (4)
(1) edge node[above, pos=0.75] {} (3)
(2) edge node[below, pos=0.33] {} (4)
(2) edge node[below] {} (3)
(4) edge node[right] {} (3)
(1) edge[bend right = 30] node[left] {$e_1$} (2)
(2) edge[bend right = 30] node[below] {$e_2$} (3)
(3) edge[bend right] node[right] {$e_3$} (4)
;
    \end{tikzpicture}
\end{tikzcd}
\caption{A graph whose matroid is a major of the graphic flag matroid given in Example~\ref{ex: graphic flag matroid}.}\label{fig: graphic major}
\end{figure}
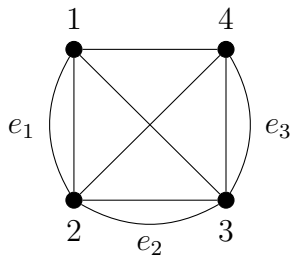

\begin{thm}\label{thm: graphic iff major is}
    A flag matroid is graphic if and only if there exists a major of it that is graphic.
\end{thm}
\begin{proof}
    Let $\mathfrak{F}$ be a graphic flag matroid with sequential representation $(M_1, \dots, M_k)$.
    Let $G$ be a graph and let $\mathcal{P}_1 \succ \dots \succ \mathcal{P}_k$ be a chain of partitions of the vertices of $G$ such that $M_i = M(G,\mathcal{P}_i)$.
    Without loss of generality, assume that each cell of $\mathcal{P}_k$ is a singleton.
    We will now construct a graph $H$ by adding edges to $G$ so that $M(H)$ is a major of $\mathfrak{F}$.

    For $i=2,\dots,k$ consider the pair $((G, \mathcal{P}_{i-1}), (G, \mathcal{P}_i))$. Every cell in $\mathcal{P}_{i-1}$ is either a cell in $\mathcal{P}_i$, or is obtained by merging some cells in $\mathcal{P}_i$ into one.
    Let $\mathcal{P}_{i-1} = \{P_1, \dots, P_l\}$ and $\mathcal{P}_i = \{P'_1, \dots P'_q\}$.
    For $j=1, \dots, l$,
    if $P_j$ is not a cell of $\mathcal{P}_{i-1}$,
    let $P'_{j_1}, \dots P'_{j_m}$ be the cells of $\mathcal{P}_i$ whose union is $P_j$.
    For each $n = 1, \dots, m$, choose a single element of $P'_{j_n}$
    and call it $v_{n,i,j}$.
    Now let $H$ be the graph obtained from $G$ by adding an edge between
    $v_{n-1}^{i,j}$ and $v_n^{i,j}$ for each $n,i,j$ such that these vertices are defined.
    Then $M_i$ is obtained from $M(H)$ by contracting each edge $v_{n-1}^{a,j}v_n^{a,j}$
    when $a > i$, and deleting it otherwise.
    So $M(H)$ is a major of $\mathfrak{F}$.

    Now suppose $\mathfrak{F}$ is a flag matroid with sequential representation $(M_1,\dots,M_k)$ and that $G$ is a graph such that $M(G)$ is a major of $\mathfrak{F}$.
    Let $X_1,\dots,X_{k-1}$ be as in Definition~\ref{defn: major}.
    For $i =1,\dots,k$, let $\mathcal{P}_i$ be the partition of the vertex set of $G$
    such that $u$ and $v$ lie in the same cell if and only if
    there is a path from $u$ to $v$ using edges from $X_{i} \cup \dots \cup X_{k-1}$.
    Then $(M(G,\mathcal{P}_1),\dots,M(G,\mathcal{P}_k))$ is a sequential representation of $\mathfrak{F}$.
\end{proof}

The non-uniqueness of majors of majors of flag matroids raises the question of whether there is a ``best'' choice of major.
Indeed, in \cite[Exercise 8.14b,c]{kung1986strong} it is mentioned that for every flag matroid $\mathfrak{F}$, there is a weak-order maximal major of $\mathfrak{F}$. However, such maximal majors of a $\mathbb{K}$-representable/graphic flag matroid need not be $\mathbb{K}$-representable/graphic.
In other words, the majors guaranteed to exist by Theorems~\ref{thm: representable iff has representable major} and~\ref{thm: graphic iff major is} need not be weak-order
maximal among all majors of a given $\mathbb{K}$-representable/graphic flag matroid.

\section{Representability}\label{section: representability}
\subsection{Binary and Ternary Flag Matroids}

A flag matroid is \emph{binary} if it is representable over $\mathbb{F}_2$ and \emph{ternary} if representable over $\mathbb{F}_3$.
Theorem~\ref{thm: minors and duals of representable} tells us that the classes of binary and ternary flag matroids are closed under taking minors.
Therefore we can, in principle, characterize such flag matroids by listing the \emph{minimally} non-binary and non-ternary flag matroids,
i.e.~the flag matroids that are non-binary (respectively, ternary) but satisfy the property that every proper minor is binary (respectively, ternary).
In this section, we do this for the classes of binary and ternary matroids that are \emph{full},
a term we now define.

\begin{defn}
    A flag matroid with sequential representation $(M_1,\dots,M_k)$ is \emph{full} if $\rank(M_{i+1}) = \rank{M_i}+1$
    for $i = 1,\dots,k-1$.
    A flag matroid $\mathfrak{F}$ is a \emph{filling} of flag matroid $\mathfrak{G}$ if $\mathfrak{F}$ is full
    and $\mathfrak{G}$ can be obtained from $\mathfrak{F}$ by a sequence of chopping operations.
\end{defn}

Existence of a filling for every flag matroid is guaranteed to exist \cite[Proposition 7.3.5]{oxley}.
That said, fillings of a flag matroid need not be unique.
Consider for example the flag matroid $\mathfrak{F}$ on $\{e_1, e_2, e_3\}$ with sequential representation $(U_{1,3},U_{3,3})$.
The flag matroid with sequential representation $(U_{1,3},U_{2,3},U_{3,3})$ is a filling of $\mathfrak{F}$,
but so is the flag matroid with sequential representation $(U_{1,3}, M, U_{3,3})$, where $M$ is the matroid
on ground set $\{e_1,e_2,e_3\}$ with bases $\{e_1,e_2\}$ and $\{e_1,e_3\}$.

\begin{prop}\label{prop: representable iff filling is}
    A flag matroid $\mathfrak{F}$ is $\mathbb{K}$-representable if and only if there is a filling of $\mathfrak{F}$ that is $\mathbb{K}$-representable.
\end{prop}
\begin{proof}
    If a filling of $\mathfrak{F}$ is $\mathbb{K}$-representable, then Theorem~\ref{thm: minors and duals of representable} implies that $\mathfrak{F}$ is as well.
    If $\mathfrak{F}$ has $\mathbb{K}$-representation $(A;d_1,\dots,d_k)$ then $\mathfrak{F}(A;d_1,d_1 + 1, \dots, d_k)$ is
    a $\mathbb{K}$-representable filling of $\mathfrak{F}$.
\end{proof}

\begin{defn}
    Let $\mathfrak{F}$ be a flag matroid on ground set $E$ with sequential representation $(M_1,\dots,M_k)$.
    A \emph{lift witness sequence for $\mathfrak{F}$} is a sequence of matroids $Q_1,\dots,Q_{k-1}$
    with $Q_i$ on ground set $E \sqcup X_i$ such that $M_i = Q_i / X_i$ and $M_{i+1} = Q_i \setminus X_i$
    and $X_i$ independent in $Q_i$.
    We call each $Q_i$ a \emph{lift witness matroid}.
\end{defn} 

A flag matroid $\mathfrak{F}$ may have multiple lift witness sequences
when $\mathfrak{F}$ is not full.
For example, the flag matroid with sequential representation $(U_{1,3}, U_{3,3})$ has multiple lift witnesses
matroids including $U_{2,5}$ and the one-element deletion of $M(K_4)$.
However, lift witness sequences for full flag matroids are unique,
as we will show after recalling two facts about matroid minors.

\begin{prop}{\cite[Corollary 3.1.24]{oxley}}\label{prop: lift pair equal iff (co)loop}
    Let $M$ be a matroid and $e \in E(M)$. Then $M/e = M\setminus e$ if and only if $e$ is a loop or coloop of $M$
\end{prop}

\begin{prop}{\cite[Proposition 3.1.27]{oxley}}\label{prop: lift witness matroid is unique}
    Let $M$ and $N$ be matroids on a common ground set $E$ and let $e \in E$. Then the following are equivalent:
    \begin{itemize}
        \item $M/e = N/e$ and $M\setminus e = N\setminus e$ 
        \item $M = N$, or $e$ is a loop of one of $M$ and $N$ and a coloop of the other
    \end{itemize}
\end{prop}

\begin{prop}\label{prop: full flag matroids have unique lift witness sequence}
    Every full flag matroid has a unique lift witness sequence up to isomorphism. 
\end{prop}
\begin{proof}
    Let $\mathfrak{F}$ be a full flag matroid with sequential representation $(M_1, \dots, M_k)$ and ground set $E$.
    Fix some $i \in \{1,\dots,k-1\}$.
    Because $M_{i+1}$ is an elementary lift of $M_i$,
    there exists a matroid $Q_i$ with ground set $E\sqcup e$ such that
    $(Q_i/e, Q_i\setminus e) = (M_i,M_{i+1})$.
    Because $M_{i+1} \neq M_i$, Proposition \ref{prop: lift pair equal iff (co)loop} implies that
    $e$ is not a loop or coloop of $Q_i$.
    Thus Proposition \ref{prop: lift witness matroid is unique} implies $Q_i$ is the unique matroid such that $M_i = Q_i/e$ and $M_{i+1} = Q_i\setminus e$.
    Since this holds for each $i \in \{1,\dots,k-1\}$, the lift witness sequence $(Q_1, \dots, Q_{i-1})$ is unique. 
\end{proof}

Recall that two representations of a matroid over a field $\mathbb{K}$ are called \emph{projectively equivalent} if one can be obtained from
the other by left-multiplying with an invertible matrix, then adding or removing linearly dependent rows.
Projective equivalence of representations of matroids will be a helpful tool for constructing representations of flag matroids.
The following fact about projective equivalence of representations of certain matroids plays a key role in the rest of the paper.

\begin{prop}{\cite[Proposition 6.6.5 and Corollary 14.6.1]{oxley}}\label{prop: binary and ternary are projectively unique}
    Let $M$ be a matroid.
    If $M$ is binary then for every field $\mathbb{K}$, all $\mathbb{K}$-representations of $M$ are projectively equivalent.
    If $M$ is ternary then all $\mathbb{F}_3$-representations of $M$ are projectively equivalent.
\end{prop}

\begin{defn}\label{defn: matroids whose K-reps are projectively equivalent}
    Given a field $\mathbb{K}$, let $\mathcal{M}(\mathbb{K})$ denote the class of $\mathbb{K}$-representable matroids $M$ such that
    all $\mathbb{K}$-representations of $M$ are projectively equivalent.
\end{defn}

If $\mathbb{K}$ is the field with either two or three elements, then Proposition~\ref{prop: binary and ternary are projectively unique}
implies that $\mathcal{M}(\mathbb{K})$ is precisely the class of $\mathbb{K}$-representable matroids.
Proposition~\ref{prop: binary and ternary are projectively unique} also implies that
for any field $\mathbb{K}$, if $M$ is $\mathbb{K}$-representable and binary, then $M \in \mathcal{M}(\mathbb{K})$. Projective uniqueness of binary and ternary matroids will be crucial in the following lemma. 

\begin{lemma}\label{lemma: proj unique K-rep iff lift witness seq is}
    Let $\mathfrak{F}$ be a flag matroid with sequential representation $(M_1,\dots,M_k)$.
    Assume that $M_i \in \mathcal{M}(\mathbb{K})$ for each $i = 1,\dots,k$
    and that $\mathfrak{F}$ has a lift witness sequence $(Q_1,\dots,Q_{k-1})$ such that each $Q_i$ is $\mathbb{K}$-representable.
    Then $\mathfrak{F}$ is $\mathbb{K}$-representable.
\end{lemma}
\begin{proof}
    Theorem~\ref{thm: representable iff has representable major} implies 
    $(M_i,M_{i+1})$ is a $\mathbb{K}$-representable flag matroid for each $i$, as $Q_i$ is a $\mathbb{K}$-representable major.
    By induction, it now suffices to show that given $\mathbb{K}$-representable flag matroids $\mathfrak{F}_1$ and $\mathfrak{F}_2$ with sequential representations $(M_1, \dots, M_{k})$ and $(M_k, M_{k+1})$ with $k \ge 2$, if each $M_i \in \mathcal{M}(K)$,
    then $(M_1, \dots, M_{k+1})$ is the sequential representation of a $\mathbb{K}$-representable flag matroid.

    Indeed, let $(A;d_1,\dots,d_k)$ be a $\mathbb{K}$-representation of $\mathfrak{F}_1$ and let $(B;r_1,r_2)$ be a $\mathbb{K}$-representation of $\mathfrak{F}_2$.
    Without loss of generality we may assume that $d_k = r_1 = \rank(M_k)$ and that $A$ has $r_1$ rows.
    Since all $\mathbb{K}$-representations of $M_k$ are projectively equivalent,
    there exists an invertible $T \in \mathbb{K}^{r_1\times r_1}$ matrix such that $A = T(B_{\le r_1})$.
    Then if $\widehat T$ is the $r_2\times r_2$ block diagonal matrix with $T$ in the upper left
    and the $(r_2-r_1)\times (r_2-r_1)$ identity in the lower right,
    i.e.
    \[
        \widehat T := \begin{pmatrix}
            T & 0 \\
            0 & I
        \end{pmatrix},
    \]
    then $(\widehat T B)_{\le r_1} = A$
    and $(\widehat T B; r_1,r_2)$ is a $\mathbb{K}$-representation of $\mathfrak{F}_2$.
    Therefore $(\widehat T B; d_1,\dots,d_k,r_2)$ is a $\mathbb{K}$-representation of $(M_1,\dots,M_k,M_{k+1})$.
\end{proof}

We are now ready to prove our
forbidden minor characterizations of binary and ternary full flag matroids.

\begin{theorem}\label{thm: full binary forbidden minors}
    A full flag matroid is binary if and only if it has no minors of the form $(U_{2,4})$ or $(U_{1,3}, U_{2,3})$,
    and ternary if and only if it has no minors of the form $(R)$ or $(R/e, R\setminus e)$ where $R \in \{U_{2,5}, U_{3,5}, F_7, F^*_7\}$
\end{theorem}
\begin{proof}
    Theorem~\ref{thm: minors and duals of representable} implies the ``only if'' direction.
    Let $\mathfrak{F}$ be a full flag matroid on ground set $E$ with sequential representation $(M_1,\dots,M_k)$.
    Let $(Q_1,\dots,Q_{k-1})$ be a lift witness sequence for $\mathfrak{F}$
    and for each $i$ let $X_i$ be independent in $Q_i$ such that $Q_{i} / X_i =M_i$ and $Q_i \setminus X_i = M_{i+1}$.
    Since $\mathfrak{F}$ is full, $X_i$ is a singleton set, so we may denote its unique element by $x_i$.

    Now suppose $\mathfrak{F}$ has no minor of the form $(U_{2,4})$ or $(U_{1,3},U_{2,3})$.
    We will show that each $Q_i$ has no minor of the form $U_{2,4}$.
    Theorem~\ref{theorem: forbidden minors finite field representability} will then imply that each $Q_i$ is binary.
    Since taking minors preserves representability, this will imply that each $M_i$ is binary,
    and therefore $M_i \in \mathcal{M}(\mathbb{F}_2)$ by Proposition~\ref{prop: binary and ternary are projectively unique}.
    Thus Lemma~\ref{lemma: proj unique K-rep iff lift witness seq is}
    will imply that $\mathfrak{F}$ is binary.
    
    Indeed, for the sake of contradiction, assume $Q_i$ has a $U_{2,4}$ minor on $\{a,b,c,d\} \subseteq E$.
    If $x_i \notin \{a,b,c,d\}$ then $\{a,b,c,d\} \subseteq E$ and $M_i$ has a $U_{2,4}$ minor on $\{a,b,c,d\}$
    contradicting our assumption that $\mathfrak{F}$ is free of $(U_{2,4})$ minors.
    Now assume without loss of generality that $x_i = d$.
    Let $T,S$ be the partition of $E \setminus \{a,b,c\}$ be such that $Q_i / T \setminus S$ is isomorphic to $U_{2,4}$.
    Then
    \begin{align*}
        &M_i / T \setminus S = Q_i / (T \cup \{x_i\}) \setminus S = U_{1,3} \qquad {\rm and}
        \\ &M_{i+1} / T \setminus S = Q_i / T \setminus (S \cup \{x_i\}) = U_{2,3}
    \end{align*}
    contradicting our assumption that $\mathfrak{F}$ is free of $(U_{1,3},U_{2,3})$ minors.
    
    A similar argument shows that if $\mathfrak{F}$ has no minor of the form $(R)$ or $(R/e, R\setminus e)$ where $R \in \{U_{2,5}, U_{3,5}, F_7, F^*_7\}$,
    then each $Q_i$ has no minor isomorphic to $U_{2,5}, U_{3,5}, F_7$ or $F^*_7$.
    As before, Theorem~\ref{theorem: forbidden minors finite field representability} then implies that each $Q_i$ is ternary.
    Then Proposition~\ref{prop: binary and ternary are projectively unique} and Lemma~\ref{lemma: proj unique K-rep iff lift witness seq is}
    imply that $\mathfrak{F}$ is ternary.
\end{proof}

Combining Theorem~\ref{thm: full binary forbidden minors} with Proposition~\ref{prop: representable iff filling is}
gives us the following.

\begin{cor}\label{cor: binary iff all fillings are binary}
    A flag matroid $\mathfrak{F}$ is binary if and only if there exists a filling of $\mathfrak{F}$ free of minors of the form $(U_{2,4})$ and $(U_{1,3},U_{2,3})$,
    and ternary if and only if there exists a filling of $\mathfrak{F}$ free of minors of the form $(R)$ or $(R/e, R\setminus e)$ where $R \in \{U_{2,5}, U_{3,5}, F_7, F^*_7\}$.
\end{cor}

The list of minimal forbidden minors for $\mathbb{F}_4$-representability is known~\cite{geelen2000excluded},
so one might wonder if they can be turned into a forbidden minor characterization for $\mathbb{F}_4$-representable
full flag matroids, similarly to the case of $\mathbb{F}_2$ and $\mathbb{F}_3$ representability.
However, in the $\mathbb{F}_4$ case we lose Proposition~\ref{prop: binary and ternary are projectively unique} since different $\mathbb{F}_4$-representations of $\mathbb{F}_4$-representable
matroids need not be projectively equivalent~\cite[Proposition~14.6.3]{oxley}.
Thus a different approach is needed to characterize $\mathbb{F}_4$-representability of flag matroids.

\subsection{Regular Flag Matroids}
We now turn our attention to \emph{regular matroids} and \emph{regular flag matroids}.

\begin{defn}
    A matrix $A\in \mathbb{Z}^{r\times n}$ is \emph{unimodular} if every maximal minor of $A$ lies in $\{-1, 0, 1\}$. A matroid $M$ is \emph{regular} if there exists a unimodular matrix $A$ such that $M=M(A)$.
\end{defn}

Although commonly defined using \emph{totally unimodular matrices}, regular matroids can be equivalently be defined using unimodular matrices, which fits into the theory of matroids over tracts with more ease. The advantage of using this definition is that it provides a nice segue into the definition of \emph{regular flag matroids} according to the theory of flag matroids over tracts\cite{jarra2024flag}. Before defining regular flag matroids, we first state the well-known characterization of regular matroids via minimal forbidden minors. 

\begin{thm}[{\cite[Theorem~6.6.6]{oxley}}]\label{thm: forbidden minors regularity}
    A matroid is regular if and only if it has no minors isomorphic to any of $U_{2,4}, F_7$ or $F_7^*$.
\end{thm}

\begin{defn}
     A flag matroid $\mathfrak{F}$ with sequential representation $(M_1, \dots, M_k)$ is \emph{regular} if there exists a unimodular matrix $A$ and a sequence of integers $1\leq d_1 < \cdots < d_k = r$ such that  $A_{\leq d_i}$ is unimodular for every $i$ and $\mathfrak{F} = \mathfrak{F}(A; d_1, \dots, d_k)$.
\end{defn}

\begin{lemma}\label{lemma: regular iff lift witness is}
    A full flag matroid is regular if and only if every matroid in its lift witness sequence is regular. 
\end{lemma}
\begin{proof}
    Let $\mathfrak{F}$ be a regular full flag matroid with sequential representation $(M_1, \dots, M_k)$, and let $A$ be a unimodular matrix such that $\mathfrak{F} = \mathfrak{F}(A)$. For a fixed $i$, we consider the pair $(M_i, M_{i+1})$. We consider now the following matrix
    \begin{equation*}
        D:= \begin{pmatrix}
            \begin{array}{c|c}
                \multirow{3}{*}{\scalebox{1.5}{$A_{\leq i+1}$}} & \mathlarger{\textbf{0}} \\
                 & \\
                & \textbf{$1$}
            \end{array}
        \end{pmatrix}
    \end{equation*}
    where $D$ is obtained by appending the unit column vector with the last entry being $1$ onto $A_{\leq i+1}$. The linear matroid of $D$, $Q_i := M(D)$, will be the unique lift witness matroid of $(M_i, M_{i+1})$. It now suffices to show that $D$ is unimodular. Any maximal minor of $D$ that does not involve the new column vector will be in $\{-1, 0, 1\}$. Any maximal minor that does involve the new column vector will necessarily still have to be in $\{-1, 0, 1\}$, as performing a Laplace expansion on the $1$ in the new unit vector will result in computing a maximal minor of $A_{\leq i}$, which is a unimodular matrix. Therefore, $Q_i$ is regular. 

    The other direction follows by Lemma \ref{lemma: proj unique K-rep iff lift witness seq is} using the fact that all representations of a regular matroid are projectively equivalent\cite[Proposition 6.6.5]{oxley}. 
\end{proof}

With Lemma \ref{lemma: regular iff lift witness is}, we can now prove that the class of regular full flag matroids is minor closed, up to some conditions on choppings. We define a \emph{non-middle chopping} to be a chopping where either the first or last matroid in the sequential representation is removed. 

\begin{prop}\label{prop: regular full flag matroids minors}
    Let $\mathfrak{F}$ be a regular full flag matroid and let $\mathfrak{F}'$ be a flag matroid obtained from $\mathfrak{F}$ by a sequence of deletions, contractions, and non-middle choppings only. Then $\mathfrak{F}'$ will be a regular full flag matroid.
\end{prop}
\begin{proof}
    Let $A$ be a unimodular matrix so that $\mathfrak{F} = \mathfrak{F}(A)$. If we consider non-middle choppings, then $C_{-1}(\mathfrak{F}) = \mathfrak{F}(A; d_1+1, d_1+2, \dots, d_1+r)-1$ and $C_{-k}(\mathfrak{F}) = \mathfrak{F}(A; d_1, d_1+1, \dots, d_1+r-2)$, and so the resulting full flag matroid will remain regular. 

    Deletions will correspond to removing columns from $A$, which preserve unimodularity. In order to prove that contractions preserve regularity, it suffices to show that $\mathfrak{F}^*$ will be regular. To do so, using Lemma \ref{lemma: regular iff lift witness is}, it suffices to show that if $(Q_1, \dots, Q_{k-1})$ is the lift witness sequence of $\mathfrak{F}$, then $(Q_{k-1}^*, \dots, Q_1^*)$ is the lift witness sequence of $\mathfrak{F}^*$, as the dual of a regular matroid will remain regular.

    Fix $i$, and let $(M_i, M_{i+1})$ be a pair in the sequential representation of $\mathfrak{F}$. This pair will have $Q_i$ as the unique lift witness matroid, and so we have $(M_i, M_{i+1}) = (Q_i/e, Q_i\setminus e)$ for some $e \in E(Q_i)$. By Proposition \ref{prop: minors sequential representation}, we obtain the following
    \begin{equation*}
        (Q_i/e, Q_i\setminus e)^* = (M_i, M_{i+1})^* = (M_{i+1}^*, M_i^*) = ((Q\setminus e)^*, (Q/e)^*) = (Q^*/e, Q^*\setminus e)
    \end{equation*}
Therefore, $Q^*$ is the unique lift witness matroid of $(M_{i+1}^*, M_i^*)$. 
\end{proof}

Because the class of regular full flag matroids is minor closed (only considering non-middle choppings), we can say that in spirit there will exist a list of minimal forbidden minors that completely determine when full flag matroids are regular. Using the same argument as in the proof of Theorem \ref{thm: full binary forbidden minors} along with Theorem \ref{thm: forbidden minors regularity}, we obtain the following result.

\begin{thm}\label{thm: regular flag matroids forbidden minors}
    A full flag matroid is regular if and only if it has no minors of the form $(R)$ or $(R/e, R\setminus e)$ where $R \in \{U_{2, 4}, F_7, F_7^*\}$.
\end{thm}

In addition to their minimal forbidden minor characterization, regular matroids also satisfy the attractive property of having multiple equivalent elegant characterizations. A matroid is regular if and only if it is representable over every field\cite[Theorem 6.6.3]{oxley}, if and only if it is binary and ternary\cite[Theorem 10.1.2]{oxley}. We can provide similar characterizations for regular full flag matroids. To do so, we will make use of the following well-known property of determinants.

\begin{lemma}\label{lemma: det commutes}
    Let $A \in R^{r\times r}$ be a matrix with entries in a ring $R$, and let $f : R\to S$ be a ring homomorphism. Then $f(\det(A)) = \det(f(A))$.
\end{lemma}

\begin{thm}\label{thm: regular flag matroids equivalences}
    Let $\mathfrak{F}$ be a full flag matroid. The following statements are equivalent. 
    \begin{enumerate}
        \item $\mathfrak{F}$ is regular.
        \item $\mathfrak{F}$ is representable over every field.
        \item $\mathfrak{F}$ is binary and ternary.
    \end{enumerate}
\end{thm}
\begin{proof}
    $(2)$ immediately implies $(3)$. If we now suppose $\mathfrak{F}$ is binary and ternary, then this implies every matroid in the lift witness sequence is binary and ternary by Lemma \ref{lemma: proj unique K-rep iff lift witness seq is}, which implies that every such matroid in the lift witness sequence is regular, which in turn implies $(1)$ by Lemma \ref{lemma: regular iff lift witness is}. 

    We now show that $(1)\Rightarrow (2)$. Consider the unique ring homomorphism $f : \mathbb{Z} \to \mathbb{K}$ for any field $\mathbb{K}$. For any unimodular representation $(A; d_1, d_1+1, \dots, d_1+r-1)$ of $\mathfrak{F}$, we have that $f(\det(A)) = \det(f(A))$ by Lemma \ref{lemma: det commutes}, which tells us that $(f(A); d_1, d_1+1, \dots, d_1+r-1)$ will be a $\mathbb{K}$-representation of $\mathfrak{F}$.
\end{proof}

\subsection{Graphic Matroids}

The list of excluded minors for graphic matroids is also known,
but our methods for $\mathbb{F}_2$ and $\mathbb{F}_3$ representability also do not generalize.
In particular, there exist non-graphic full flag matroids satisfying the property that every matroid in the lift witness sequence is graphic.

\begin{prop}\label{prop: lift witness sequence graphic does not imply flag matroid graphic}
    There exists a full flag matroid $\mathfrak{F}$ such that every member of its lift witness sequence is graphic, but $\mathfrak{F}$ is not graphic. 
\end{prop}
\begin{proof}
     Let $H_1,H_2,G_2,$ and $G_3$ be the graphs shown in Figure~\ref{fig: graphs}
    and observe that $G_2$ is formed from $G_3$ by identifying the red vertices of $G_3$ together,
    that $M(G_2) = M(H_2)$, and $H_1$ is formed by identifying the two red vertices of $H_2$ together. 
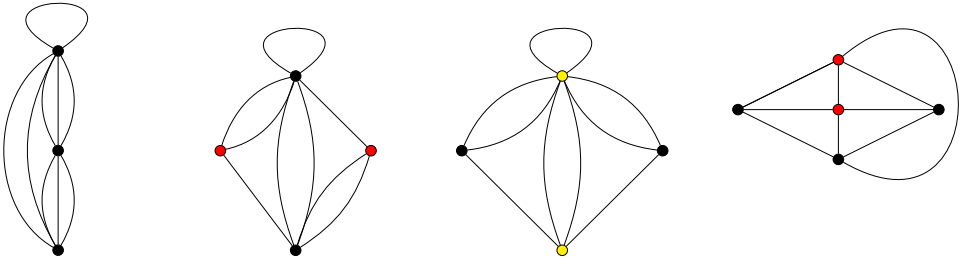
\begin{figure}[H]
    \centering
\resizebox{14cm}{!}{
\begin{tikzpicture}
            \vertex[fill] (1) at (2,-1) {};
  \vertex[fill] (2) at (2,1)  {};
  \vertex[fill] (3) at (2,3) {};
  \node[] at (0.60, 0) {};
  \path (3) edge [out=150,in=35,looseness=30] node[above] {} (3)
  (2) edge node[above, pos=0.4] {} (3)
  (1) edge[bend left = 30] node[above, pos=0.7] {} (2)
  (1) edge[bend right = 30] node[below, pos=0.8] {} (2)
  (1) edge[] node[] {} (2)
  (2) edge[bend left = 30] node[above, pos=0.7] {} (3)
  (2) edge[bend right = 30] node[below, pos=0.8] {} (3)
  (1) edge[bend left = 30] node[] {} (3)
  (1) edge[bend left = 60] node[] {} (3)
  ;
\end{tikzpicture}
\hspace{5mm}
\begin{tikzpicture}
  \vertex[fill] (1) at (0, 4) {};
  \vertex[fill] (2) at (0,0.5) [] {};
  \vertex[fill=red, label=left:$1$] (3) at (-1.5, 2.5){};
  \vertex[fill=red, label=right:$2$] (4) at (1.5, 2.5) [] {};
  \path 
  (1) edge [out=150,in=35,looseness=30] node[above] {} (1)
  (2) edge[bend left = 20] node[] {} (1)
  (2) edge[bend right = 20] node[] {} (1)
  (2) edge[bend left = 20] node[above, pos=0.4] {} (4)
  (2) edge[bend right = 20] node {} (4)
  (2) edge node[] {} (3)
  (3) edge[bend left = 30] node[] {} (1)
  (3) edge[bend right = 30] node[] {} (1)
  (4) edge[] node[] {} (1)
  ;
\end{tikzpicture}
\hspace{10mm}
\begin{tikzpicture}
  \vertex[fill=yellow, label=above:$3$] (1) at (0, 4) {};
  \vertex[fill=yellow, label=below:$4$] (2) at (0,0.5) [] {};
  \vertex[fill] (3) at (-2, 2.5){};
  \vertex[fill] (4) at (2, 2.5) [] {};
  \path (1) edge [out=150,in=35,looseness=30] node[above] {} (1)
  (2) edge[bend left = 20] node[] {} (1)
  (2) edge[bend right = 20] node[] {} (1)
  (2) edge node[above, pos=0.4] {} (4)
  (2) edge node[] {} (3)
  (3) edge[bend left = 30] node[] {} (1)
  (3) edge[bend right = 30] node[] {} (1)
  (4) edge[bend left = 30] node[] {} (1)
  (4) edge[bend right = 30] node[] {} (1)
  ;
\end{tikzpicture}
\hspace{10mm}
\begin{tikzpicture}
  \vertex[fill] (1) at (5,1) [] {};
  \vertex[fill] (2) at (7,2) [] {};
  \vertex[fill=red, label=above:$5$] (3) at (5,3) [] {};
  \vertex[fill] (4) at (3,2) [] {};
  \vertex[fill=red, label=below left:$6$] (5) at (5,2) [] {};
    \path
    (1) edge node[left] {} (2) 
    (1) edge node[below] {} (5)
    (3) edge node[below] {} (5)
    (1) edge node[above, pos=0.75] {} (4)
    (2) edge node[below, pos=0.33] {} (3)
    (2) edge node[below] {} (5)
    (4) edge node[below] {} (5)
    (3) edge node[right] {} (4)
    (3) edge node {} (4)
    (1) edge[out = -30, in = 40, looseness = 4.5] node {} (3)
    ;
\end{tikzpicture}
}
\caption{From left to right, the graphs $H_1,H_2,G_2,G_3$.}\label{fig: graphs}
\end{figure}
For the sake of clarity, we will refer to vertices $1$ and $2$ in $H_2$ as \emph{red vertices}, vertices $3$ and $4$ in $G_2$ will be referred to as \emph{yellow vertices}, and vertices $5$ and $6$ in $G_3$ will be referred to as \emph{red vertices}. All other vertices for any graph here will be referred to as a \emph{black vertex}.

Define $M_1 := M(H_1)$, $M_2 := M(H_2) \cong M(G_2)$, and $M_3 := M(G_3)$.
Then $(M_1, M_2, M_3)$ is the sequential representation of a full flag matroid $\mathfrak{F}$.
Let $N_1$ be the matroid of the graph obtained from $H_2$ by adding an edge between the red vertices, and let $N_2$ be the matroid of the graph obtained from $G_3$ by adding a second edge between the red vertices.
Then $(N_1,N_2)$ is the lift witness sequence of $\mathfrak{F}$.

We now show that $\mathfrak{F}$ is not graphic. 
Suppose there exists a graph $L$ and a chain of partitions $\mathcal{P}_1 \succ \mathcal{P}_2 \succ \mathcal{P}_3$
such that $M_i = M(L, \mathcal{P}_i)$ for $i = 1,2,3$.
Because $G_3$ is $3$-connected, it must be the case that $(L, \mathcal{P}_3) = G_3$ by \cite[Lemma~5.3.2]{oxley}.
So assume $L = G_3$ and that $\mathcal{P}_3$ is the partition of the vertices of $G_3$ into singletons.

We now show that any other graph formed from $G_3$ by identifying any two vertices in a way that does not result in $G_2$
is a graph whose cycle matroid is different from $M(G_2)$,
thus implying that $(G_3, \mathcal{P}_2) = G_2$.
Indeed, by symmetry,
it suffices to only consider the cases where the top red vertex in $G_3$ is identified with the left black vertex,
and where the left and right black vertices are identified together.
Let $G_3^{rb}$ denote the graph obtained from $G_3$ by identifying the top red vertex with the left black vertex,
and let $G_3^{bb}$ denote the graph obtained from $G_3$ by identifying the two black vertices.
These graphs are shown in Figure~\ref{fig: Gbb and Grb}.
Indeed, $M(G_3^{bb})$ has no loops whereas $M_2$ does,
and $M(G_3^{rb})$ has only two parallel classes whereas $M_2$ has three.
Thus we must have $(G_3,\mathcal{P}_2) = G_2$.

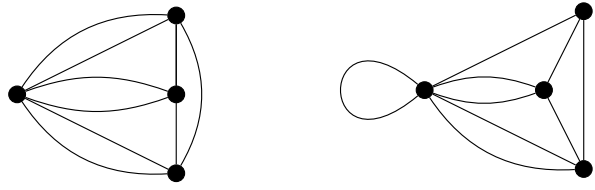
\begin{figure}[H]
    \centering
\resizebox{8cm}{!}{
\begin{tikzpicture}
  \vertex[fill] (2) at (0, 2) {};
  \vertex[fill] (3) at (0,0) [] {};
  \vertex[fill] (1) at (-2, 1){};
  \vertex[fill] (4) at (0, 1) {};
  \path 
  (1) edge[] node[] {} (2)
  (1) edge[bend left = 30] node[] {} (2)
  (2) edge node[above, pos=0.4] {} (4)
  (2) edge node[] {} (3)
  (3) edge[] node[] {} (1)
  (3) edge[bend left = 30] node[] {} (1)
  (4) edge[bend left = 20] node[] {} (1)
  (4) edge[bend right = 20] node[] {} (1)
  (2) edge[bend left = 30] node[] {} (3)
  ;
\end{tikzpicture}
\hspace{10mm}
\begin{tikzpicture}
  \vertex[fill] (1) at (5,-1) [] {};
  \vertex[fill] (2) at (4.5,0) [] {};
  \vertex[fill] (3) at (3,0) [] {};
  \vertex[fill] (4) at (5,1) [] {};
    \path
    (1) edge[] node[] {} (4) 
    (1) edge[] node[] {} (2)
    (1) edge[] node[] {} (3) 
    (1) edge[bend left = 30] node[] {} (3)
    (4) edge[] node[] {} (2)
    (4) edge[] node[] {} (3)
    (2) edge[bend left = 20] node[] {} (3)
    (2) edge[bend right = 20] node[] {} (3)
    (3) edge[out = 140, in = 220, looseness = 30] node[] {} (3)
     ;
\end{tikzpicture}
}
\caption{From left to right, the graphs $G_3^{bb}$ and $G_3^{rb}$.}\label{fig: Gbb and Grb}
\end{figure}


On the graph level, $(G_3, \mathcal{P}_1)$ is formed from $(G_3, \mathcal{P}_2)$ by identifying two vertices together.
If a yellow vertex in $G_2$ is identified with a black vertex, the resulting graph would have three loops, and so the resulting graphic matroid would not be $M_1$.
A similar situation arises if the two yellow vertices are identified together.
If the two black vertices are identified together, then the resulting matroid would have four elements in parallel, whereas $M(H_1)$ does not.
Therefore, there is no vertex identification in $G_2$ that results in a graph whose cycle matroid equals $M_1$, and so $\mathfrak{F}$ cannot be graphic. 
\end{proof}

\section{Acknowledgments}
We would like to thank Changxin Ding and Donggyu Kim for the productive discussion regarding regular full flag matroids resulting in the forbidden minor characterization, along with the productive discussion regarding full graphic flag matroids and for providing the example in Figure~\ref{fig: graphs}.

\bibliographystyle{abbrv}
\bibliography{flagmatroids}
\end{document}